\documentclass[11pt, a4paper]{article}

\usepackage{amsmath,comment}
\usepackage{amssymb}
\usepackage{amsthm}
\usepackage{bm}
\usepackage{graphicx}
\usepackage{psfrag}
\usepackage{enumerate}
\usepackage{mathrsfs}

\newtheorem{thm}{Theorem}[section]

\newtheorem{lem}[thm]{Lemma}

\newtheorem{assumption}[thm]{Assumption}
\newtheorem{algorithm}[thm]{Algorithm}

\newtheorem{definition}[thm]{Definition}

\newtheorem{example}[thm]{Example}

\newtheorem{remark}[thm]{Remark}
\newenvironment{rem}{\begin{remark}\rm}{\end{remark}}
\newtheorem{tab}{Table}

\title{The Minimum Number of Hubs in Networks}

\author{\small \begin{tabular}{ccc}
Easton Li Xu & Guangyue Han\\
The University of Hong Kong & The University of Hong Kong\\
email: xulimc@gmail.com & email: ghan@hku.hk\\
\end{tabular}}

\date{\empty}

\begin{document} \maketitle

\begin{abstract}
In this paper, a hub refers to a non-terminal vertex of degree at least three. We study the minimum number of hubs needed in a network to guarantee certain flow demand constraints imposed between multiple pairs of sources and sinks. We prove that under the constraints, regardless of the size or the topology of the network, such minimum number is always upper bounded and we derive tight upper bounds for some special parameters. In particular, for two pairs of sources and sinks, we present a novel path-searching algorithm, the analysis of which is instrumental for the derivations of the tight upper bounds.
\end{abstract}

\smallskip
\section{Introduction}

Consider a network $G=(V, E)$, where $V$ denotes the set of vertices in $G$, and $E$ denotes the set of edges in $G$. A vertex in $G$ is said to be a {\it source} if it is only incident with outgoing edges, and a {\it sink} if it is only incident with incoming edges. Often, a source or sink is referred to as a {\it terminal} vertex. A non-terminal vertex is said to be a {\it hub} if its degree is greater than or equal to $3$. In this paper, we are primarily concerned with the minimum number of hubs needed when certain constraints on the flow demand between multiple pairs of sources and sinks are imposed. The flow demand constraints considered in this paper will be in terms of the vertex-cuts between pairs of sources and sinks. This can be justified by a vertex version of the max-flow min-cut theorem~\cite{be1998}, which states that for a network with infinite edge-capacity and unit vertex-capacity, the maximum flow between one source and one sink is equal to the minimum vertex-cut between them. Here, we remark that with appropriately modified setup, our results can be stated in terms of edge-cuts as well.

More precisely, for given $C_1, C_2, \ldots, C_n \in \mathbb{N}$, let $\mathcal{N}(C_1, C_2, \ldots, C_n)$ denote the set of all finite networks $G$ (see Figure~\ref{illustrator} for an example) such that
\begin{itemize}
\item there are $n$ sources $S_1, S_2, \ldots, S_n$ and $n$ sinks $R_1, R_2, \ldots, R_n$ in $G$;
\item all edges in $G$, except those incident with a source or sink, are undirected (alternatively, bi-directional);
\item for each feasible $i$, the minimum vertex-cut from $S_i$ to $R_i$ is $C_i$.
\end{itemize}
Now, we define
$$
\mathscr{H}(C_1, C_2, \ldots, C_n) \triangleq \sup_{G \in \mathcal{N}(C_1, C_2, \ldots, C_n)} \min_{\substack{\widehat{G} \subset G \\ \widehat{G} \in \mathcal{N}(C_1, C_2, \ldots, C_n)}} \mathcal{H}(\widehat{G}),
$$
where $\mathcal{H}(\widehat{G})$ denotes the number of hubs in $\widehat{G}$. The above definition can be roughly interpreted as follows: for a given $G$, we try to find a subgraph $\widehat{G}$ that contains the minimum number of hubs required to satisfy the vertex-cut constraints, and $\mathscr{H}(C_1,C_2,\ldots,C_n)$ gives us the minimum number corresponding to the worst-case scenarios among all possible $G$.

\begin{figure}
  \centering
  \includegraphics[width=0.3\textwidth]{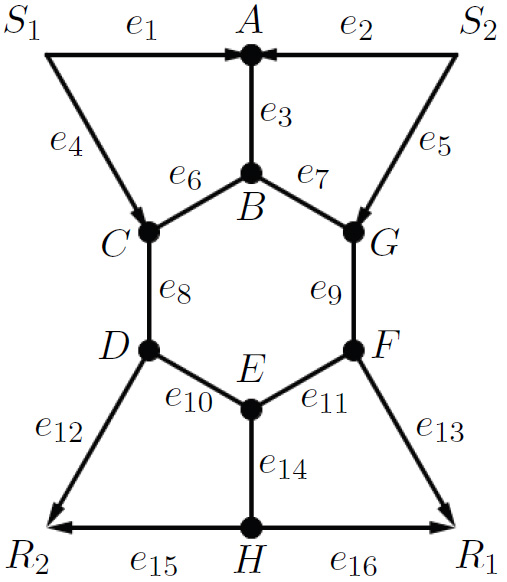}
  \caption{An illustrative graph in $\mathcal{N}(2,2)$ with $8$ hubs.}\label{illustrator}
\end{figure}
At first glance, $\mathscr{H}(C_1, C_2, \ldots, C_n)$ can be infinite. One of the main (and somewhat surprising) results in this paper, Theorem~\ref{finiteness}, however, states that for any given $C_1, C_2, \ldots, C_n$, $\mathscr{H}(C_1, C_2, \ldots, C_n)$ is in fact finite. With finiteness confirmed, we are primarily interested in computing the value of $\mathscr{H}(C_1, C_2, \ldots, C_n)$. We say a graph $G$ is $\emph{minimal}$ in $\mathcal{N}(C_1,C_2,\ldots,C_n)$, if $G \in \mathcal{N}(C_1, C_2, \ldots, C_n)$, but for any $e \in E$, $G\backslash\{e\} \not \in \mathcal{N}(C_1,C_2,\ldots,C_n)$. It then follows from the fact that every graph in $\mathcal{N}(C_1, C_2, \ldots, C_n)$ has at least one minimal subgraph that
$$
\mathscr{H}(C_1, C_2, \ldots, C_n) = \sup_{\substack{G \in \mathcal{N}(C_1, C_2, \ldots, C_n)\\ G\textrm{ is minimal}}} \mathcal{H}(G).
$$

The vertex-connectivity version of the classical Menger's theorem~\cite{me1927} states that for a network with one pair of source $S$ and sink $R$ with the minimum vertex-cut between them being $C$, there exist $C$ vertex-disjoint paths connecting $S$ and $R$, which immediately implies that $\mathscr{H}(C)=0$ for any given $C$. Theorem~\ref{finiteness} states that for any given $G \in \mathcal{N}(C_1, C_2, \ldots, C_n)$, one can always find a subgraph $\widehat{G}$ of $G$ such that $\mathcal{H}(\widehat{G})$ is upper bounded by a constant, which is independent of the choice of $G$. In some sense, Theorem~\ref{finiteness} can be viewed as a generalization of the vertex-connectivity version of Menger's theorem.

Mathematically, the proposed problem of computing $\mathscr{H}(C_1, C_2, \ldots, C_n)$ is a natural combinatorial optimization problem. On a more practical side, hubs in networks naturally correspond to more costly vertices. For instance, in a transportation network, as opposed to ``relaying'' vertices with degree $2$, hubs may have to be better equipped for traffic scheduling; for this reason, when designing the route map, an airline may need to avoid running too many airline hubs to reduce the cost. So, as might be expected, $\mathscr{H}(C_1, C_2, \ldots, C_n)$ is of significance to cost-minimizing resource allocation in transportation networks.

To the best of our knowledge, the proposed problem of computing or estimating $\mathscr{H}$ has not yet been examined previously and there is little related work in the vast literature of graph theory. On the other hand, to a great extent, this work is motivated by the study of network encoding complexity (see~\cite{la2006} and references therein), where the number of encoding vertices in directed networks is of primary concern. Moreover, our approaches to tackle the problem are influenced by those in network encoding complexity theory, particularly, to a greater extent, those in~\cite{ha2011, xsh2012}.

The remainder of the paper is organized as follows. In Section~\ref{minimalgraph}, we give necessary and sufficient conditions (Theorem~\ref{equivalent}) for a graph being minimal in $\mathcal{N}(C_1, C_2)$. In Section~\ref{sectionminimalrepresentation}, we introduce the notion of a representation of a graph in $\mathcal{N}(C_1, C_2)$ and we present the structural decomposition theorem (Theorem~\ref{decomposition}) for representations of minimal graphs in $\mathcal{N}(C_1, C_2)$. We will introduce in Section~\ref{xu-algorithm} a novel path-searching algorithm, the analysis of which will aptly produce an upper bound on $\mathscr{H}(C_1, C_2)$ for any given $C_1, C_2$. In Section~\ref{sectionhmn}, we derive the value of $\mathscr{H}(C_1,C_2)$ (Theorem~\ref{Nmn2mn}), which is a main result in this paper. Another main result is Theorem~\ref{finiteness}, which establishes the finiteness of $\mathscr{H}(C_1, C_2, \ldots, C_n)$, $n \geq 3$, through a recursive bounding argument. The remaining part of Section~\ref{more-than-two} will be devoted to the derivations of the values of $\mathscr{H}$ for some special parameters (Theorem~\ref{thm11mn}, \ref{thmh222}).

\smallskip
\section{Minimal Graphs in $\mathcal{N}(C_1,C_2)$}\label{minimalgraph}

We first introduce some notation and terminologies that will be used throughout the paper.

A sequence of edges in $G$, $e_i=(u_i, v_i)$, $i=1, 2, \ldots, d$, with $v_i=u_{i+1}$, $i=1, 2, \ldots, d-1$, can be linked to form a {\it path} $p$, denoted by $p = e_1 \circ e_2 \circ \cdots \circ e_d$; furthermore, $p$ is called a {\it cycle} if $v_d=u_1$. A path $p$ of this form is said to be {\it directed} each $e_i$ is oriented such that all of them concatenate.

For a path $\beta$ and two vertices $u$, $v$ of $\beta$, let $\beta[u,v]$ denote the subpath of $\beta$ between $u$ and $v$. For a directed path $\beta=e_1 \circ e_2 \circ \cdots \circ e_d$, let $h(\beta)$, $t(\beta)$ denote the head, tail of $\beta$, respectively; and we say $e_i$ is $\emph{smaller}$ than $e_j$ on $\beta$, denoted by $e_i < e_j$, if $h(e_j)$ is the head of the directed path $\beta[e_i, e_j]$ (or alternatively, $e_j$ is {\it larger} than $e_i$, denoted by $e_j > e_i$).

For a graph $G \in \mathcal{N}(C_1, C_2, \ldots, C_n)$, by the vertex-connectivity version of Menger's theorem, for each $i$, one can find a set $\alpha_i$ of $C_i$ vertex-disjoint $\phi$-paths from $S_i$ to $R_i$. If the choice of $\alpha_i$ is unique, $\alpha_i$ is said to be $\emph{non-reroutable}$, otherwise it is said to be {\it reroutable}. $G$ is said to be {\it non-reroutable} if all $\alpha_i$'s are non-reroutable, {\it reroutable} otherwise. For any index set $\{j_1, j_2, \ldots, j_k\} \subset \{1, 2, \ldots, n\}$, let $G_{\alpha_{j_1} \alpha_{j_2} \cdots \alpha_{j_k}}$ denote the subgraph of $G$ induced on the edges of $\alpha_{j_1}, \alpha_{j_2}, \ldots, \alpha_{j_k}$-paths. $G$ is said to be a {\it $(C_1, C_2, \ldots, C_n)$-graph} if $G=\bigcup_{i=1}^n G_{\alpha_i}$, that is, each edge in $G$ belongs to some $\alpha_i$-path. In order to compute $\mathscr{H}(C_1, C_2, \ldots, C_n)$, it is enough to only consider all $(C_1, C_2, \ldots, C_n)$-graphs, since $\bigcup_{i=1}^n G_{\alpha_i}$ is a subgraph of $G$ and also in $\mathcal{N}(C_1,C_2,\ldots,C_n)$.

Sections~\ref{minimalgraph} to \ref{sectionhmn} will be devoted to derive the value of $\mathscr{H}(C_1, C_2)$. For notational convenience only, we often rewrite $\alpha_1$, $\alpha_2$ as $\phi$, $\psi$, respectively.

For a $(C_1, C_2)$-graph $G$, an edge in $G$ is said to be \emph{public} if it is shared by a $\phi$-path and a $\psi$-path, \emph{private} otherwise. Evidently, for each $i$, from $S_i$ to $R_i$, each $\phi$ or $\psi$-path in $G$ induces a natural orientation to all its edges. We note that a public edge in $G$ may have opposite $\phi$-direction and $\psi$-direction (such ``inconsistency'' will be dealt with in Section~\ref{sectionminimalrepresentation}).

We say $G$ can be \emph{naturally orientable} if each public edge in $G$ has the same natural $\phi$ and $\psi$-direction. A cycle $e_1\circ e_2\circ \cdots \circ e_d$ in $G$, where $e_i=(u_i,v_i)$, $v_i=u_{i+1}$ for $i=1,2,\ldots,d-1$ and $v_d=u_1$, is said to be a \emph{$\phi$-consistent cycle}, if it satisfies the following property: for any $1\le i\le d$, if $e_i$ belongs to a $\phi$-path, then its natural $\phi$-direction is from $u_i$ to $v_i$. And we similarly define a $\psi$-consistent cycle.

The following theorem gives necessary and sufficient conditions for a $(C_1, C_2)$-graph being minimal.
\begin{thm}\label{equivalent}
The following three statements are equivalent for a $(C_1, C_2)$-graph $G$:
\begin{itemize}
\item[\bf (i)] $G$ is minimal;
\item[\bf (ii)] $G$ is non-reroutable;
\item[\bf (iii)] $G$ has no $\phi$ or $\psi$-consistent cycle.
\end{itemize}
\end{thm}

\begin{proof}

\textbf{1.} \textbf{(ii) $\bm{\Rightarrow}$ (i)}:
Any edge $e$ in $G$ must belong to some $\phi$ or $\psi$-path. After deleting $e$ from $G$, we no longer find $C_i$ vertex-disjoint paths from $S_i$ to $R_i$ for some $i\in \{1,2\}$. So $G\backslash \{e\}\not \in \mathcal{N}(C_1,C_2)$, and therefore $G$ is minimal.

\textbf{2.} \textbf{(ii) $\bm{\Rightarrow}$ (iii)}:
Suppose, by way of contradiction, that $G$ has a $\phi$-consistent cycle $O$, which can be written in the following form
$$p_1 \circ e_1 \circ p_2 \circ e_2\circ \cdots \circ p_d\circ e_d,$$
where each $p_i$ is a subpath of some $\phi$-path, and each $e_i$ is a private $\psi$-edge. Furthermore, we assume $O$ has the smallest $d$ (the number of private $\psi$-edges) among all $\phi$-consistent cycles. Then, each $p_i$ belongs to a different $\phi$-path (since otherwise we can always find a $\phi$-consistent cycle with fewer private $\psi$-edges), which further implies that $d\le C_1$. Suppose $p_i=\phi_i[u_i,v_i]$ for $1\le i\le d$, $e_i=(v_i,u_{i+1})$ for $1\le i\le d-1$ and $e_d=(v_d,u_1)$. Then one can find another group of $C_1$ vertex-disjoint paths $\hat{\phi}=\{\hat{\phi}_1,\hat{\phi}_2,\ldots,\hat{\phi}_{C_1}\}$ from $S_1$ to $R_1$ in $G$, where
\begin{displaymath}
\hat{\phi}_i=\begin{cases}
\phi_1[S_1,u_1] \circ (u_1,v_d) \circ \phi_d[u_d,R_1]&\textrm{ for }i=1;\\
\phi_i[S_1,u_i] \circ (u_i,v_{i-1}) \circ \phi_{i-1}[v_{i-1},R_1]&\textrm{ for }2\le i\le d;\\
\phi_i&\textrm{ for }d+1\le i\le C_1,\\
\end{cases}
\end{displaymath}
which contradicts the assumption that $G$ is non-reroutable. With a parallel argument, we conclude that $\psi$-consistent cycles do not exist either.

\textbf{3.} \textbf{(i) or (iii) $\bm{\Rightarrow}$ (ii)}:
Suppose, by contradiction, that $G$ is reroutable, and furthermore, by symmetry, that there exists another group of $C_1$ vertex-disjoint paths $\hat{\phi}=\{\hat{\phi}_1,\hat{\phi}_2,\ldots,\hat{\phi}_{C_1}\}$ from $S_1$ to $R_1$ in $G$ with $\hat{\phi}_i$ sharing the same outgoing edge from $S_1$ as $\phi_i$ for every $i$. Pick a $\hat{\phi}$-path, say, $\hat{\phi}_{i_1}$, such that $\hat{\phi}_{i_1} \neq \phi_{i_1}$, and let $v_{i_1}$ denote the smallest vertex on $\phi_{i_1}$ (under the natural $\phi$-direction) where they leave each other. Assume that, after $v_{i_1}$, $\hat{\phi}_{i_1}$ first meets some $\phi$-path, say, $\phi_{i_2}$ at the vertex $u_{i_1}$. Denote by $v_{i_2}$ the smallest vertex where $\hat{\phi}_{i_2}$ and $\phi_{i_2}$ leave each other. Assume that, after $v_{i_2}$, $\hat{\phi}_{i_2}$ first meets some $\phi$-path, say, $\phi_{i_3}$ at the vertex $u_{i_2}$. Continue the procedure in a similar manner to obtain an index sequence $i_1,i_2,\ldots,i_t,\ldots$, and similarly define $v_{i_t}$'s and $u_{i_t}$'s. Pick the smallest $k$ such that $i_k=i_j$ for some $j < k$. Notice that $v_{i_{t+1}}$ is smaller than $u_{i_t}$ on $\phi_{i_{t+1}}$, which easily follows from three facts:
\begin{itemize}
\item[1)] $\hat{\phi}_{i_{t+1}}$ first takes apart from $\phi_{i_{t+1}}$ at $v_{i_{t+1}}$;
\item[2)] $\hat{\phi}_{i_t}$ meets $\phi_{i_{t+1}}$ at $u_{i_t}$;
\item[3)] $\hat{\phi}_{i_t}$ and $\hat{\phi}_{i_{t+1}}$ are vertex-disjoint.
\end{itemize}
Hence, we conclude that
\begin{align*}
\nonumber &\hat{\phi}_{i_{j}}[v_{i_{j}},u_{i_{j}}]
\circ \phi_{i_{j+1}}[u_{i_j},v_{i_{j+1}}]
\circ \hat{\phi}_{i_{j+1}}[v_{i_{j+1}},u_{i_{j+1}}]
\circ \phi_{i_{j+2}}[u_{i_{j+1}},v_{i_{j+2}}]
\circ \\
\cdots &\circ \hat{\phi}_{i_{k-2}}[v_{i_{k-2}},u_{i_{k-2}}]
\circ \phi_{i_{k-1}}[u_{i_{k-2}},v_{i_{k-1}}]
\circ \hat{\phi}_{i_{k-1}}[v_{i_{k-1}},u_{i_{k-1}}]
\circ \phi_{i_{j}}[u_{i_{k-1}},v_{i_{j}}]
\end{align*}
is a $\phi$-consistent cycle in $G$.

Since all $\psi$-paths are vertex-disjoint and the above cycle does not contain any terminal vertex, at least one edge $e$ in this cycle does not belong to any $\psi$-path (that is, $e$ is a private $\phi$-edge). Notice that each edge of $\hat{\phi}_{i_{t}}[v_{i_{t}},u_{i_{t}}]$, $j\le t\le k-1$, is a private $\psi$-edge, and each edge of $\phi_{i_{t+1}}[u_{i_t},v_{i_{t+1}}]$, $j\le t\le k-1$, does not belong to any $\hat{\phi}$-path. So, $e$ must belong to $\phi_{i_{t+1}}[u_{i_t},v_{i_{t+1}}]$ for some $t$ with $j\le t\le k-1$, and hence $e$ does not belong to any $\hat{\phi}$-path. In the graph $G\backslash\{e\}$, we can find a set $\hat{\phi}$ of $C_1$ vertex-disjoint paths from $S_1$ to $R_1$, and a set $\psi$ of $C_2$ vertex-disjoint paths from $S_2$ to $R_2$. Thus $G\backslash\{e\}\in \mathcal{N}(C_1,C_2)$ and $G$ is not minimal, a contradiction.
\end{proof}

\begin{remark}\label{ifreroutable}
The third part of the proof of the theorem has actually proved that for a $(C_1,C_2)$-graph $G$, if $\phi$ (resp. $\psi$) is reroutable, then there exists a private $\phi$ (resp. $\psi$)-edge $e$ such that $G\setminus\{e\}\in \mathcal{N}(C_1,C_2)$. This fact will be used later in the paper.
\end{remark}

\smallskip
\section{Representations and Structural Decomposition} \label{sectionminimalrepresentation}

In this section, we will transform a minimal $(C_1,C_2)$-graph $G$ into $G^\circ$, its representation, through the following steps:

{\bf Step 1 [Remove Relays]}: In this step, we remove all non-terminal vertices in $G$ with degree $2$. In more detail, for any non-terminal vertex $v$ with $d(v)=2$ and $e_1=(v,u_1)$, $e_2=(v,u_2)$ being the two edges incident with $v$, where $u_1\ne u_2$, delete edges $e_1$, $e_2$ and vertex $v$, and then add a new edge $(u_1,u_2)$. Let $G_1$ denote the resulting graph.

{\bf Step 2 [Stretch Crossings]}: We say a $\phi$-path and a $\psi$-path {\it cross} at vertex $v$ if they share $v$, but not any edges incident with $v$. In this step, we convert each crossing into a pair of degree $3$ vertices. In more detail, for any vertex $v$ in $G_1$ with $d(v)=4$ and $e_1=(v,u_1)$, $e_2=(v,u_2)$ being the two $\phi$-edges incident with $v$, $e_3=(v,u_3)$, $e_4=(v,u_4)$ being the two $\psi$-edges incident with $v$, where all $u_i$ are all distinct, delete edges $e_1$, $e_2$, $e_3$, $e_4$ and vertex $v$, and then add two vertices $v_1$, $v_2$ and edges $(v_1,v_2)$, $(u_1, v_1)$, $(u_3, v_1)$, $(u_2, v_2)$ and $(u_4, v_2)$. Let $G_2$ denote the resulting graph.

{\bf Step 3 [Match Directions]}: For any public edge $e$ in $G_2$ with inconsistent $\phi$ and $\psi$-direction, we will perform the following operations to obtain consistency: Assume edge $e=(u,v)$ belongs to both path $\phi_i=\phi_i[S_1,w_1] \circ(w_1,u)\circ (u,v)\circ(v,w_2)\circ \phi_i[w_2,R_1]$ and path $\psi_j=\psi_j[S_2,w_3]\circ(w_3,v)\circ(v,u)\circ(u,w_4)\circ\psi_j[w_4,R_2]$. We delete edges $(w_3,v),(u,w_4)$, add edges $(w_3,u),(v,w_4)$, and then obtain a new $\psi$-path
$$
\psi_j[S_2,w_3]\circ (w_3,u)\circ (u,v)\circ (v,w_4)\circ \psi_j[w_4,R_2].
$$
Let $G^\circ$ denote the resulting graph, which, evidently, is naturally orientable; let $\overrightarrow{G}^{\circ}$ denote the directed version of $G^{\circ}$, equipped with the consistent natural orientation induced on all its $\phi$ and $\psi$-paths. Apparently, $G^\circ \in \mathcal{N}(C_1, C_2)$ with its $\phi$ and $\psi$-paths determined by the original ones, and all the non-terminal vertices in $G^{\circ}$ are hubs.

The obtained $G^\circ$ after the above three steps is said to be a \emph{representation} of $G$. The following lemma states that $G^{\circ}$ must be a minimal $(C_1, C_2)$-graph as well.
\begin{lem}\label{minimalrepresentation}
The representation $G^{\circ}$ of a minimal $(C_1, C_2)$-graph $G$ is also minimal.
\end{lem}

\begin{proof}
First, since $G$ is minimal, by Theorem~\ref{equivalent}, $G$ is non-reroutable. Evidently, $G_1$ is also non-reroutable, and thus minimal. By way of contradiction, suppose that $G^\circ$ is not minimal. Again, by Theorem~\ref{equivalent}, $G^\circ$ is reroutable. Notice that in Step 2 and 3, both crossings or inconsistently oriented public edges are converted into consistently oriented public edges. Then, by Remark~\ref{ifreroutable}, one can find a private edge $e$ such that $G^\circ\backslash\{e\}\in \mathcal{N}(C_1,C_2)$, implying $G_1\backslash\{e\}$ also belongs to $\mathcal{N}(C_1,C_2)$, which contradicts the minimality of $G_1$.
\end{proof}

Let $\mathcal{N}^\circ(C_1, C_2)$ denote the subset of $\mathcal{N}(C_1,C_2)$ consisting of all networks $G$ such that $G$ is minimal, naturally orientable and all the non-terminal vertices in $G$ are of degree 3. Apparently, $\mathcal{N}^\circ(C_1,C_2)$ is the set of the representations of all minimal $(C_1,C_2)$-graphs. The following theorem says that in order to compute $\mathscr{H}(C_1, C_2)$, it is enough to only consider all graphs in $\mathcal{N}^\circ(C_1, C_2)$.
\begin{thm}
$$
\mathscr{H}(C_1, C_2) = \sup_{G \in \mathcal{N}^\circ(C_1, C_2)} \mathcal{H}(G).
$$
\end{thm}

\begin{proof}
The ``$\leq$'' direction follows from the observation that
$$
\mathcal{H}(G)=\mathcal{H}(G_1)\le \mathcal{H}(G_2)=\mathcal{H}(G^\circ),
$$
and the ``$\ge$'' direction immediately follows from Lemma~\ref{minimalrepresentation}.
\end{proof}

A path in $G^\circ \in \mathcal{N}^\circ(C_1, C_2)$ is said to be an {\it alternating} path if all its edges are privates edges and the terminal pair of this path is one of the following: $(S_1, S_2)$, $(S_1, R_1)$, $(R_2, S_2)$ or $(R_2, R_1)$.

\begin{lem}
An alternating path has the following properties:
\begin{enumerate}
\item Each of its $\phi$-edges is only adjacent to $\psi$-edges, and each of its $\psi$-edges is only adjacent to $\phi$-edges.
\item Each of its $\phi$ (resp. $\psi$)-edge belongs to a different $\phi$ (resp. $\psi$)-path.
\end{enumerate}
\end{lem}

\begin{proof}
\textbf{1.} This follows from the fact that after Step $1$, vertices with degree $2$ have been removed and thus no two private $\phi$ (or $\psi$)-edges are adjacent.

\textbf{2.} We show that in any alternating path, each $\phi$-edge belongs to a different $\phi$-path. Suppose, by contradiction, that for an alternating path path $e_1 \circ e_2 \circ \cdots \circ e_d$, where $e_i=(u_i,v_i)$, two edges $e_k, e_l$, $k < l$, both belong to $\phi_t$. If $e_l$ is smaller than $e_k$ on path $\phi_t$ in $\overrightarrow{G}^\circ$, then
$$
\phi_t[u_l,v_k]\circ e_{k+1} \circ e_{k+2} \circ \cdots \circ  e_{l-1}
$$
is a $\phi$-consistent cycle in $G^\circ$, which, by Theorem~\ref{equivalent}, gives us a contradiction. If $e_k$ is smaller than $e_l$ on path $\phi_t$ in $\overrightarrow{G}^\circ$, then the subpath $\phi_t[v_k,u_l]$ in fact can be expressed as
$f_1\circ f_2 \circ \cdots \circ f_{2p-1}$, where $f_1, f_3, \ldots, f_{2p-1}$ are public and $f_2, f_4,\ldots, f_{2p-2}$ are private. Then
$$(e_{k+1}\circ f_1) \circ f_2 \circ \cdots \circ f_{2p-2}\circ (f_{2p-1}\circ e_{l-1})\circ e_{l-2}\circ\cdots\circ e_{k+2}$$
is a $\psi$-consistent cycle in $G^\circ$, which, by Theorem~\ref{equivalent}, gives us a contradiction.
With a parallel argument, we conclude that each private $\psi$-edge belongs to a different $\psi$-path.
\end{proof}

We next present the structural decomposition theorem of a representation $G^{\circ} \in \mathcal{N}^\circ(C_1, C_2)$, which, roughly speaking,  states that after deleting public edges in $G^{\circ}$, each connected component in the resulting graph is an alternating path. More precisely, letting $G^{\circ}_{\textrm{p}}$ denote the subgraph of $G^\circ$ induced on all private edges in $G^{\circ}$, we have
\begin{thm}\label{decomposition}
$G^{\circ}_{\textrm{p}}$ consists of $(C_1+C_2)$ alternating paths.
\end{thm}

\begin{proof}
Consider $\overrightarrow{G}^{\circ}_{\textrm{p}}$, the naturally oriented version of $G^{\circ}_{\textrm{p}}$. Obviously, the degree of any non-terminal vertex in $\overrightarrow{G}^{\circ}_{\textrm{p}}$ is $2$. Now, starting from $S_1$, traverse along an outgoing $\phi$-edge, say, $e_1$, and then traverse against the $\psi$-edge adjacent to $e_1$, say, $e_2$, and then along a $\phi$-edge adjacent to $e_2$, say, $e_3$, and then  against a $\psi$-edge adjacent to $e_3$ \ldots. Continue the procedure in this fashion, we will always reach $R_1$ or $S_2$, since otherwise the set of edges that we have traversed will contain a cycle, which is both $\phi$-consistent and $\psi$-consistent. Evidently, a similar argument can be applied to the case when we start from an incoming edge incident with $R_2$. It then follows that one can find a set of $(C_1+C_2)$ vertex-disjoint paths in $G^{\circ}_{\textrm{p}}$ from $\{S_1,R_2\}$ to $\{S_2,R_1\}$; moreover, it can easily checked that each edge in $G^{\circ}_{\textrm{p}}$ belong to one of the above-mentioned vertex-disjoint paths.
\end{proof}

\begin{rem} \label{head-tail}
Let $v$ be a non-terminal vertex of an alternating path in $G^{\circ}_{\textrm{p}}$ and let $e, e'$ be the two edges incident with $v$. It then follows from the proof of Theorem~\ref{decomposition} that if $v$ is the head (resp. tail) of $e$ in $\overrightarrow{G}^{\circ}_{\textrm{p}}$, then it is also the head (resp. tail) of $e'$.
\end{rem}

An alternating path in $G^{\circ}_{\textrm{p}}$ is said to be an $S_1 S_2$, $S_1 R_1$, $R_2 S_2$, $R_2 R_1$-alternating path if its terminal pair is $(S_1, S_2)$, $(S_1, R_1)$, $(R_2, S_2)$ or $(R_2, R_1)$, respectively. A path is said to be an $S_1$-alternating path is it is either an $S_1 S_2$-alternating path or $S_1 R_1$-alternating path, similarly,  a path is also referred to as an $R_2$-alternating path is it is either an $R_2 S_2$-alternating path or $R_2 R_1$-alternating path. Apparently, there are $C_1$ $S_1$-alternating paths and $C_2$ $R_2$-alternating paths. For two vertices $u,v$ of an $S_1$ (resp. $R_2$)-alternating path $L$, we say $u$ is on the right of $v$ in $L$ if $v$ is ``nearer'' to $S_1$ (resp. $R_2$) than $u$ in $L$ (more precisely, the number of edges between $u$ and $S_1$ (resp. $R_2$) in $L$ is more than that between $v$ and $S_1$ (resp. $R_2$)) (the word ``right'' arises since we will ``position'' the vertices of a path in a plane in Section~\ref{xu-algorithm} for an easy illustration). For two edges $e,e'$ in $L$, we say $e$ is on the right of $e'$ if one of two vertices incident with $e$ is on the right of the two ones incident with $e'$ in $L$.

\begin{example}
Figure~\ref{illustrator} shows a naturally oriented minimal $(2,2)$-graph with $12$ hubs. From $S_1$ to $R_1$, there are two vertex-disjoint $\phi$-paths $e_1\circ e_3\circ e_7\circ e_9\circ e_{13}$ and $e_4\circ e_8\circ e_{10}\circ e_{14}\circ e_{16}$; and from $S_2$ to $R_2$, there are two vertex-disjoint $\psi$-paths $e_2\circ e_3\circ e_6\circ e_8\circ e_{12}$ and $e_5\circ e_9\circ e_{11}\circ e_{14}\circ e_{15}$. The edges $e_3$, $e_8$, $e_9$, $e_{14}$ are public since each of them is shared by some $\phi$-path and some $\psi$-path. The others are all private. Then we can find that $G^{\circ}_{\textrm{p}}$ consists of four alternating paths: two $S_1$-alternating paths
$$e_1\circ e_2,\ e_4\circ e_6\circ e_7\circ e_5,$$
and two $R_2$-alternating paths
$$e_{15}\circ e_{16},\ e_{12}\circ e_{10}\circ e_{11}\circ e_{13}.$$
Moreover, $e_4\circ e_6\circ e_7\circ e_5$ is an $S_1S_2$-alternating path, where edge $e_7$ is on the right of edge $e_6$, vertex $B$ is on the right of vertex $C$.
\end{example}

\smallskip
\section{A Path-Searching Algorithm} \label{xu-algorithm}

In this section, we introduce an algorithm, the analysis of which will be instrumental for deriving the value of $\mathscr{H}(C_1, C_2)$. Before rigorously describing the algorithm, we roughly illustrate its idea.

Consider a representation $G^{\circ}$ of a minimal $(C_1,C_2)$-graph $G$. Imagine on a two-dimensional plane, each alternating path $e_1\circ e_2\circ \cdots \circ e_d$ in $G^{\circ}_{\textrm{p}}$ is ``positioned'' into a ``double deck'' (see Figure~\ref{alternatingpath} for an example), where $e_{i+1}$ is on the right of $e_i$, and all the tails (in $\overrightarrow{G}_p^{\circ}$) are on the upper deck and heads on the lower deck (see Remark~\ref{head-tail}). A vertex on the upper deck is referred to as an \emph{upper vertex} if it is not a source, and a vertex on the lower deck is referred to as a $\emph{lower vertex}$ if it is not a sink. Notice that there is one more lower vertices than upper vertices in an $S_1S_2$-alternating path, while one less in an $R_2R_1$-alternating path, and there are equally many of them in an $S_1R_1$-alternating path or an $R_2S_2$-alternating path.

\begin{figure}[h]
\centering
\includegraphics[width=0.5\textwidth]{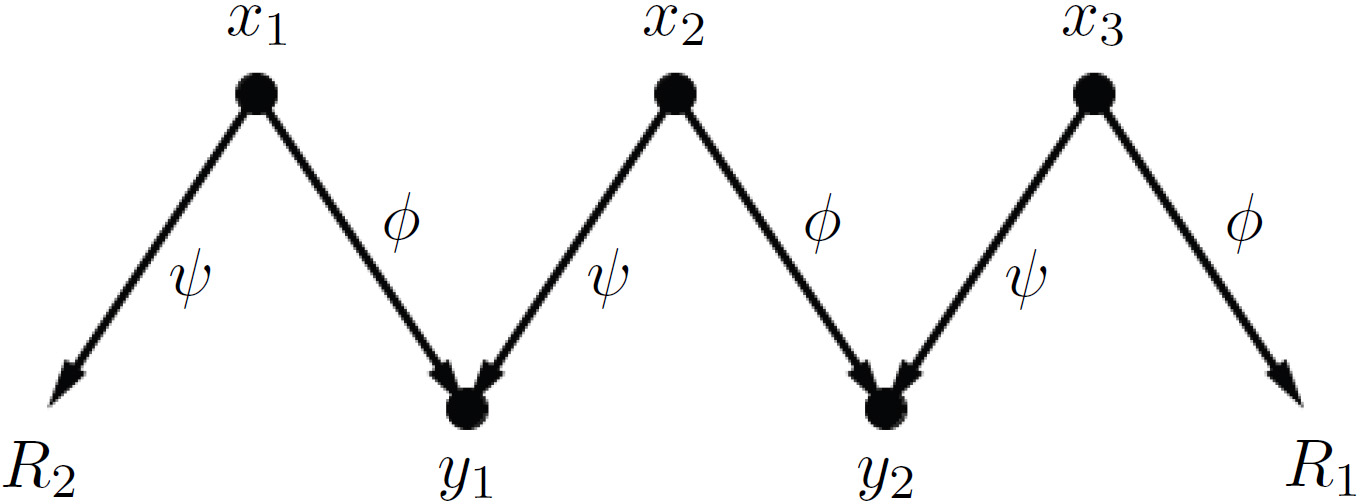}
\caption{In this $R_2R_1$-alternating path, $x_1$, $x_2$, $x_3$ are upper vertices and $y_1$, $y_2$ are lower vertices.}\label{alternatingpath}
\end{figure}

In each $S_1S_2$-alternating (resp. $R_2R_1$-alternating) path, one of the lower (resp. upper) vertices will be labeled as a $\emph{choke vertex}$ (which could be a ``bottleneck'' for the path extending procedure in the algorithm), which is initially the rightmost lower (resp. rightmost upper) vertex of the path and will be updated during the execution of  the algorithm.

As summarized below, the algorithm will iteratively find a set of the so-called \emph{interconnecting} paths in $G^{\circ}$, which will ``link'' the double decks.
\begin{itemize}
\item In the beginning of each iteration, we initialize the interconnecting path as a public edge, whose tail is an unoccupied lower vertex. See STEP 2.
\item We first traverse along edges in $\overrightarrow{G}^{\circ}$ to extend the interconnecting paths to link the double decks. When we reach a choke vertex, we check if it is possible to switch paths (see Figure~\ref{switching}) and find new ways to extend the interconnecting paths. See STEP 3(A), 3(B).
\item Then we traverse against edges in $\overrightarrow{G}^{\circ}$ to extend the interconnecting paths in a parallel manner. See STEP 5(A), 5(B).
\item At the end of $k$-th iteration, a set of $k$ interconnecting paths is found, whose vertices are all labeled as occupied. See STEP 6.
\end{itemize}
As proven later, the algorithm will produce all interconnecting paths and upon termination of the algorithm, all hubs in $G^{\circ}$ are occupied.

Let $\Delta$ be the number of $S_1S_2$-alternating paths, then $\Delta\le \min\{C_1,C_2\}$, and moreover, the numbers of $R_2R_1$-alternating paths, $S_1R_1$-alternating paths and $R_2S_2$-alternating paths are $\Delta, C_1-\Delta, C_2-\Delta$, respectively. The following algorithm will find $\Delta$ interconnecting paths in $G^{\circ}$.

\bigskip
\begin{algorithm}\label{algorithm}
Input: $G^{\circ}$; Output: a set $\mathcal{I}$ of $\Delta$ interconnecting paths.
\end{algorithm}

\begin{description}
\item[STEP 1: Initialize the algorithm.] \

Orient each edge in the direction of $\phi$ and $\psi$-paths;

Label all hubs as ``unoccupied'';

\textbf{FOR} each $S_1S_2$-alternating path \textbf{DO}

\qquad Its choke vertex $:=$ its rightmost lower vertex;

\textbf{FOR} each $R_2R_1$-alternating path \textbf{DO}

\qquad Its choke vertex $:=$ its rightmost upper vertex;

$\mathcal{I} := \emptyset$;

$n := 1$.

\item[STEP 2: Initialize the forward extension.] \

$v :=$ an arbitrarily picked unoccupied lower vertex and label $v$ as ``occupied'';

$f :=$ the public edge whose tail is $v$;

$u := h(f)$ and label $u$ as ``occupied'';

$P := f$.

\item[STEP 3(A): Prepend a private edge.] \

$L :=$ the alternating path to which $u$ belongs;

\textbf{IF} $L$ is an $R_2R_1$-alternating path \textbf{AND} $u$ is the choke vertex of $L$ \textbf{THEN} \textbf{BEGIN}

\qquad \textbf{IF} there are no unoccupied upper vertices in $L$ \textbf{THEN}

\qquad \qquad Go to \textbf{STEP 4};

\qquad \textbf{ELSE} \textbf{BEGIN}

\qquad \qquad $x_0:=$ the rightmost unoccupied upper vertex of $L$;

\qquad \qquad $(x_0,y_0)\circ (y_0,x_1)\circ \cdots \circ (x_d,y_d) \circ (y_d,u) :=$ the subpath of $L$ between $x_0$ and $u$;

\qquad \qquad \textbf{IF} $d=0$ \textbf{THEN} \textbf{BEGIN}

\qquad \qquad \qquad The choke vertex of $L := x_0$;

\qquad \qquad \qquad Label $y_0$ as ``occupied'';

\qquad \qquad \qquad $e:=(u,y_0)$;

\qquad \qquad \qquad $P:= P\circ e$;

\qquad \qquad \textbf{END}

\qquad \qquad \textbf{ELSE} (\emph{viz.} $d>0$) \textbf{BEGIN}

\qquad \qquad \qquad \textbf{FOR} $i:= 1$ \textbf{TO} $d$ \textbf{DO}

\qquad \qquad \qquad \qquad $P_i:=$ the interconnecting path in $\mathcal{I}$ containing $(x_i,y_i)$;

\qquad \qquad \qquad $\mathcal{I} := \mathcal{I} \setminus \{P_1, P_2, \ldots, P_d\}$;

\qquad \qquad \qquad The choke vertex of $L := x_0$;

\qquad \qquad \qquad Label $y_0$ as ``occupied'';

\qquad \qquad \qquad $e := (x_1,y_0)$;

\qquad \qquad \qquad $\widehat{P} := P$;

\qquad \qquad \qquad $P := P_{1}[t(P_{1}), x_1] \circ e$;

\qquad \qquad \qquad \textbf{FOR} $i:= 1$ \textbf{TO} $d-1$ \textbf{DO}

\qquad \qquad \qquad \qquad $P_{i}:= P_{{i+1}}[t(P_{{i+1}}), x_{i+1}] \circ (x_{i+1},y_i) \circ P_{i}[y_i,h(P_{i})]$;

\qquad \qquad \qquad $P_{d}:= \widehat{P}[t(\widehat{P}),u] \circ (u,y_d) \circ P_{d}[y_d,h(P_{d})]$;

\qquad \qquad \qquad $\mathcal{I} := \mathcal{I} \cup \{P_1, P_2, \ldots, P_d\}$;

\qquad \qquad \textbf{END}

\qquad \textbf{END}

\textbf{END}

\textbf{ELSE} \textbf{BEGIN}

\qquad \textbf{IF} $u\in$ an $S_1$-alternating path \textbf{THEN}

\qquad \qquad $e :=$ the private $\psi$-edge whose tail is $u$;

\qquad \textbf{ELSE} (\emph{viz.} $u \in$ an $R_2$-alternating path)

\qquad \qquad $e :=$ the private $\phi$-edge whose tail is $u$;

\qquad Label $h(e)$ as ``occupied'';

\qquad $P := P \circ e$.

\textbf{END}

\item[STEP 3(B): Prepend a public edge.] \

$f :=$ public edge whose tail is $h(e)$;

$u := h(f)$ and label $u$ as ``occupied'';

$P := P \circ f$;

Go to \textbf{STEP 3(A)}.

\item[STEP 4: Initialize the backward extension.]\

$\mathcal{I} := \mathcal{I} \cup \{P\}$;

$P :=$ the interconnecting path in $\mathcal{I}$ whose tail is $v$;

$\mathcal{I} := \mathcal{I} \setminus \{P\}$;

$w := v$.

\item[STEP 5(A): Append a private edge.] \

$L := $ the alternating path to which $w$ belongs;

\textbf{IF} $L$ is an $S_1S_2$-alternating path \textbf{AND} $w$ is the choke vertex of $L$ \textbf{THEN} \textbf{BEGIN}

\qquad \textbf{IF} there is no unoccupied lower vertex in $L$ \textbf{THEN}

\qquad \qquad Go to \textbf{STEP 6};

\qquad \textbf{ELSE} \textbf{BEGIN}

\qquad \qquad $y_0:=$ the rightmost unoccupied lower vertex of $L$;

\qquad \qquad $(y_0,x_0)\circ (x_0,y_1)\circ \cdots \circ (y_d,x_d)\circ (x_d,w)\hspace{-0.8mm}:=\hspace{-0.8mm}$ the subpath of $L$ between $y_0$ and $w$;

\qquad \qquad \textbf{IF} $d=0$ \textbf{THEN} \textbf{BEGIN}

\qquad \qquad \qquad The choke vertex of $L := y_0$;

\qquad \qquad \qquad Label $x_0$ as ``occupied'';

\qquad \qquad \qquad $e:=(x_0,w)$;

\qquad \qquad \qquad $P:= e \circ P$;

\qquad \qquad \textbf{END}

\qquad \qquad \textbf{ELSE} (\emph{viz.} $d>0$) \textbf{BEGIN}

\qquad \qquad \qquad \textbf{FOR} $i:= 1$ \textbf{TO} $d$ \textbf{DO}

\qquad \qquad \qquad \qquad $P_i :=$ the interconnecting path in $\mathcal{I}$ containing $(x_i,y_i)$;

\qquad \qquad \qquad $\mathcal{I} := \mathcal{I} \setminus \{P_1, P_2, \ldots, P_d\}$;

\qquad \qquad \qquad The choke vertex of $L := y_0$;

\qquad \qquad \qquad Label $x_0$ as ``occupied'';

\qquad \qquad \qquad $e :=(x_0,y_1)$;

\qquad \qquad \qquad $\widehat{P} := P$;

\qquad \qquad \qquad $P := e\circ P_{1}[y_1, h(P_{1})]$;

\qquad \qquad \qquad \textbf{FOR} $i:= 1$ \textbf{TO} $d-1$ \textbf{DO}

\qquad \qquad \qquad \qquad $P_{i}:= P_{i}[t(P_{i}),x_i] \circ (x_i,y_{i+1}) \circ P_{{i+1}}[y_{i+1}, h(P_{{i+1}})]$;

\qquad \qquad \qquad $P_{d} := P_{d}[t(P_{d}),x_d] \circ (x_d, w) \circ \widehat{P}[w, h(\widehat{P})]$;

\qquad \qquad \qquad $\mathcal{I} := \mathcal{I} \cup \{P_1, P_2, \ldots, P_d\}$;

\qquad \qquad \textbf{END}

\qquad \textbf{END}

\textbf{END}

\textbf{ELSE} \textbf{BEGIN}

\qquad \textbf{IF} $w \in$ an $R_2$-alternating path \textbf{THEN}

\qquad \qquad $e :=$ the private $\phi$-edge whose head is $w$;

\qquad \textbf{ELSE} (\emph{viz.} $w \in$ an $S_1$-alternating path)

\qquad \qquad $e :=$ the private $\psi$-edge whose head is $w$;

\qquad Label $t(e)$ as ``occupied'';

\qquad $P:= e\circ P$.

\textbf{END}

\item[STEP 5(B): Append a public edge.]\

$f :=$ public edge whose head is $t(e)$;

$w := t(f)$ and label $w$ as ``occupied'';

$P := f \circ P$;

Go to \textbf{STEP 5(A)}.

\item[STEP 6: Terminate the iteration.]\

\textbf{IF} $n=\Delta$ \textbf{THEN}

\qquad Terminate the algorithm;

$\mathcal{I} := \mathcal{I} \cup \{P\}$;

$n := n+1$;

Go to \textbf{STEP 2}.

\end{description}

\begin{figure}[h]
\centering
\includegraphics[width=0.53\textwidth]{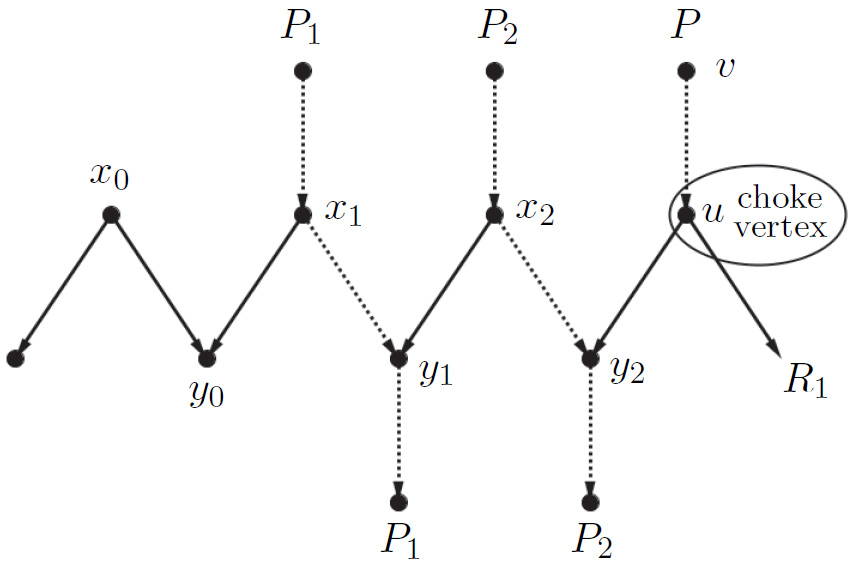}

before switching paths
\vskip 0.5cm

\includegraphics[width=0.53\textwidth]{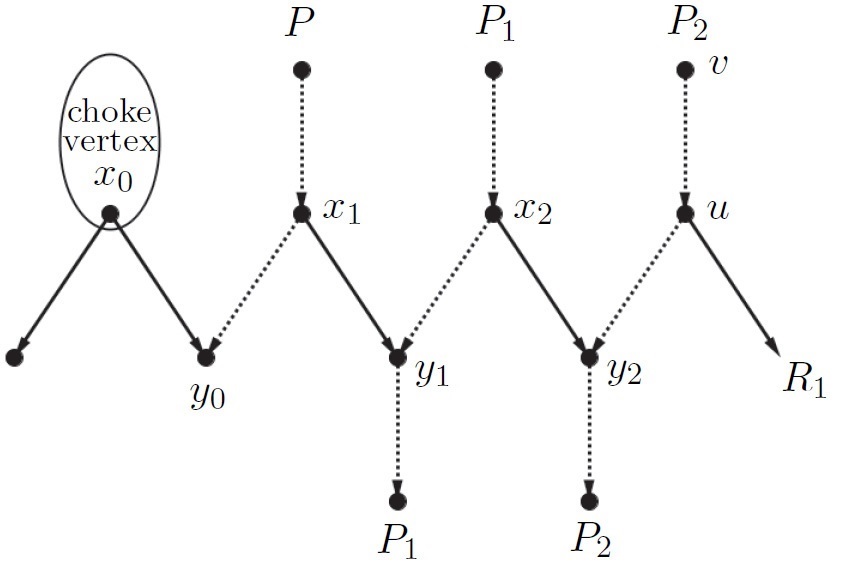}

after switching paths

\caption{Switch interconnecting paths when a choke vertex is met.}\label{switching}
\end{figure}

\begin{rem}
\textbf{1.} For Step 2 in each iteration, an unoccupied lower vertex always exists since there is at least one $S_1S_2$-alternating path with its chock vertex left unoccupied.

\textbf{2.} Each iteration of Algorithm~\ref{algorithm} consists of forward and backward extensions, and each interconnecting path starts from an $S_1S_2$-alternating path and ends at an $R_2 R_1$-alternating path. The algorithm cannot be simplified into a ``one-direction extension'' version: for each iteration, the extending procedure will terminate at an $S_1S_2$-alternating path with only one unoccupied lower vertex, but, in the beginning of each iteration, such an alternating path may not exist (to see this, notice that more than two unoccupied vertices of an alternating path can become ``occupied'', if we switch paths during one iteration).
\end{rem}

We will need the following two lemmas in the next section.

\begin{lem}\label{distinctAAsequence}
For a representation $G^{\circ}$ of a minimal $(C_1,C_2)$-graph, upon the termination of Algorithm~\ref{algorithm}, each private edge of any interconnecting path belongs to a different alternating path.
\end{lem}

\begin{proof}
By contradiction, suppose that for an interconnecting path $P=e_0\circ e_1\circ e_2\circ \cdots \circ e_{2n}$, two private edges $e_{2k}$ and $e_{2l}$, $k<l$, belong to the same alternating path $L$. Let $O$ be the cycle $$P[h(e_{2k}),t(e_{2l})]\circ L[t(e_{2l}),h(e_{2k})].$$
If $e_{2l}$ is on the right of $e_{2k}$ in $L$, then $O$ is a $\psi$-consistent cycle in $G^{\circ}$. If $e_{2k}$ is on the right of $e_{2l}$ in $L$, then $O$ is a $\phi$-consistent cycle in $G^{\circ}$. By Theorem~\ref{equivalent}, both cases imply that $G^{\circ}$ is not minimal, which, by Lemma~\ref{minimalrepresentation}, further implies that $G$ is not minimal, a contradiction.
\end{proof}

\begin{lem}\label{atmostmconnectingpaths}
For a representation $G^{\circ}$ of a minimal $(C_1,C_2)$-graph $G$, upon the termination of Algorithm~\ref{algorithm}, $\mathcal{I}$ consists of at most $\min\{C_1,C_2\}$ interconnecting paths, and each hub in $G^{\circ}$ is in exactly one of the interconnecting paths in $\mathcal{I}$.
\end{lem}

\begin{proof}
Each interconnecting path in $\mathcal{I}$ starts from a lower vertex of a distinct $S_1S_2$-alternating path and ends to an upper vertex of a distinct $R_2R_1$-alternating path. So the number of interconnecting paths is $\Delta$ with $\Delta\le \min\{C_1,C_2\}$. Now upon the termination of Algorithm~\ref{algorithm}, pick an unoccupied lower vertex, and then execute the algorithm from STEP 3(B); or pick an unoccupied upper vertex, and then execute the algorithm from STEP 3(A). Since all chock vertices have been unoccupied, the algorithm will fail to terminate, violating the fact that $G^\circ$ is finite.
\end{proof}

\section{$\mathscr{H}(C_1, C_2)$} \label{sectionhmn}

\begin{thm}\label{Nmn2mn}
$$
\mathscr{H}(C_1,C_2)= 2C_1C_2.
$$
\end{thm}

\begin{proof}
\textbf{The ``$\le$'' direction:} For a representation $G^{\circ}$ of a minimal $(C_1,C_2)$-graph $G$, apply Algorithm~\ref{algorithm} to obtain a set $\mathcal{I}$ of $\Delta$ interconnecting paths. In the $t$-th iteration of the algorithm, the forward extension stops at the choke vertex of some $R_2R_1$-alternating path. Let $L_t$ denote this alternating path. Note that after $t$ iterations, 1) all hubs of $L_t$ become occupied; 2) by Lemma~\ref{distinctAAsequence}, each of the $t$ obtained interconnecting paths contains at most two hubs of $L_t$; 3) one obtained interconnecting path contains exactly one hub of $L_t$. Hence, we deduce that the number of hubs of $L_t$ is at most $2t-1$, and therefore the total number of hubs of all $R_2R_1$-alternating paths is at most $\sum_{t=1}^{\Delta} (2t-1)=\Delta^2$. Similarly, the total number of hubs of all $S_1S_2$-alternating paths is at most $\Delta^2$ as well. Again, by Lemma~\ref{distinctAAsequence}, the total number of hubs of all $S_1R_1$ and $R_2S_2$-alternating paths is at most $2\Delta(C_1+C_2-2\Delta)$. By Lemma~\ref{atmostmconnectingpaths}, each hub in $G^{\circ}$ belongs to some interconnecting path in $\mathcal{I}$. Therefore,
$$|\mathcal{H}(G^\circ)|\le 2\Delta^2+2\Delta(C_1+C_2-2\Delta)=2\Delta(C_1+C_2-\Delta)\le 2C_1C_2,$$
where the last inequality follows from $\Delta\le \min\{C_1,C_2\}$.

\textbf{The ``$\ge$'' direction:} We only need to construct a minimal $(C_1,C_2)$-graph $G$ with $2C_1C_2$ hubs (see Figure~\ref{picn33} for an example). The graph $G$ can be described as follows:  $G \in \mathcal{N}^\circ(C_1, C_2)$ is naturally oriented, and there is a set of $C_1$ vertex-disjoint paths $\phi=\{\phi_1,\phi_2,\ldots,\phi_{C_1}\}$ from $S_1$ to $R_1$ and a set of $C_2$ vertex-disjoint paths $\psi=\{\psi_1,\psi_2,\ldots,\psi_{C_2}\}$ from $S_2$ to $R_2$. Furthermore, in $\overrightarrow{G}$, the directed version of $G$, paths $\phi_i$ and $\psi_j$ meet at vertex $\lambda_{i,j}$ and depart at vertex $\mu_{i,j}$, for $1\le i\le C_1$, $1\le j\le C_2$, and
\begin{itemize}
\item on path $\phi_i$,
$\lambda_{i,1}<\mu_{i,1}<\lambda_{i,2}<\mu_{i,2}<\cdots<\lambda_{i,C_2}<\mu_{i,C_2};$
\item on path $\psi_j$,
$\lambda_{1,j}<\mu_{1,j}<\lambda_{2,j}<\mu_{2,j}<\cdots<\lambda_{C_1,j}<\mu_{C_1,j}.$
\end{itemize}
\end{proof}

\begin{figure}[h]
\centering
\includegraphics[width=0.4\textwidth]{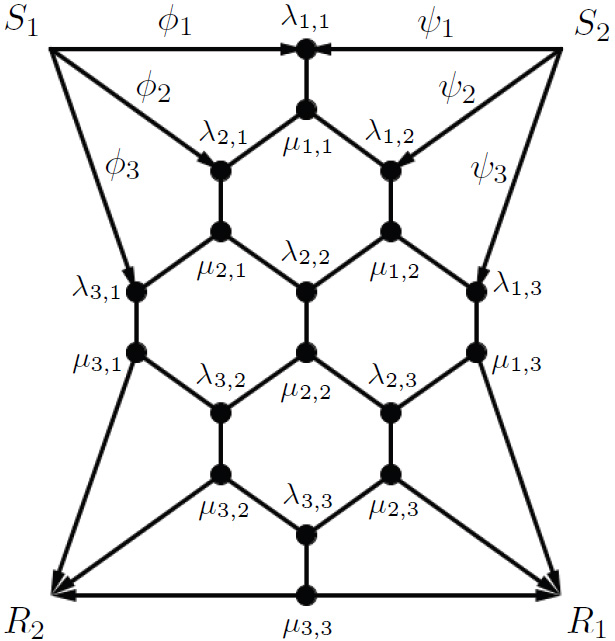}
\caption{A minimal $(3,3)$-graph with $18$ hubs.}\label{picn33}
\end{figure}

\section{$\mathscr{H}(C_1, C_2, \ldots, C_n)$} \label{more-than-two}

In this section, we are concerned with $\mathscr{H}$ with more than two parameters, which turns out to be much more difficult to compute than $\mathscr{H}$ with two parameters.

The following theorem establishes the finiteness of $\mathscr{H}$ with more than two parameters.
\begin{thm} \label{finiteness}
For any given $C_1, C_2, \ldots, C_n\in \mathbb{N}$,
$$
\mathscr{H}(C_1, C_2, \ldots, C_n) < \infty.
$$
\end{thm}

\begin{proof}

We will use an inductive argument on $n$. Notice that the case when $n=2$ has been established in Theorem~\ref{Nmn2mn}. By way of induction, we assume that the theorem is true for $n \leq k-1$ and proceed to prove it for $n=k$.

Consider any minimal $(C_1, C_2, \ldots, C_k)$-graph $G$. Let $\alpha_i$ denote the set of $C_i$ vertex-disjoint paths from $S_i$ to $R_i$ for $1\le i\le k$. After necessary rerouting of $\alpha_1, \alpha_2, \ldots, \alpha_{k-1}$ within $G_{\alpha_1 \alpha_2 \cdots \alpha_{k-1}}$, we can assume that $G_{\alpha_1 \alpha_2 \cdots \alpha_{k-1}}$ is minimal and thus
$$
N_1\triangleq \mathcal{H}\left(G_{\alpha_1 \alpha_2 \cdots \alpha_{k-1}}\right) \leq \mathscr{H}(C_1, C_2, \ldots, C_{k-1}) < \infty.
$$

Let $\widehat{G}$ denote the subgraph of $G_{\alpha_1 \alpha_2 \cdots \alpha_{k-1}}$ induced on all the edges, each of which is simultaneously an $\alpha_1, \alpha_2, \ldots, \alpha_{k-1}$-edge, and let $\omega$ denote the number of connected components in $\widehat{G}$. Obviously, each connected component of $\widehat{G}$ is in fact a path, and therefore we have $\omega \le N_1/2$.

In this proof, we say a hub in $G$ is {\it new} if it is a hub in $G$, however, not one in  $G_{\alpha_1 \alpha_2 \cdots \alpha_{k-1}}$. And we say a new hub is \emph{global}, if this hub belongs to $\widehat{G}$, \emph{local}, if this hub is in $G\setminus \widehat{G}$.

Then, by the induction hypothesis, for any distinct $j_1, j_2, \ldots, j_{k-2} \in \{1, 2, \ldots, k-1\}$, we deduce that the number of new hubs in $G_{\alpha_{j_1}\alpha_{j_2} \cdots \alpha_{j_{k-2}} \alpha_k} \backslash \widehat{G}$ is upper bounded by $\mathscr{H}(C_{j_1}, C_{j_2}, \ldots, C_{j_{k-2}}, C_k)$. So the number of local new hubs is at most
$$N_2\triangleq \sum_{i=1}^{k-1} \mathscr{H}(C_1,\ldots,C_{i-1},C_{i+1},\ldots,C_{k-1},C_k)<\infty.$$

For each $\alpha_{k}$-path, we ``cut'' it at each of its local new hubs and then ``divide'' the path into ``segments'', each of which, evidently, is a subpath of the original $\alpha_k$-path. Let $\breve{\alpha}_k$ denote the set of all obtained subpaths. Then we have $N_3 \triangleq |\breve{\alpha}_k| \le C_k+N_2$.

From the subgraph $\widehat{G} \cup \breve{\alpha}_k$, we construct an $(\underbrace{\omega,\omega,\ldots,\omega}_{k-1},N_3)$-graph $G'$ through the following procedure:
\begin{enumerate}
\item Add sources $S'_1,S'_2,\ldots,S'_k$, sinks $R'_1,R'_2,\ldots,R'_k$.
\item For any $i \in \{1, 2, \ldots, k-1\}$ and for any connected component of $\widehat{G}$, whose natural $\alpha_i$-direction is from one end vertex, say $v_1$, to the other end vertex, say $v_2$, add two directed edges $(S'_i,v_1)$ and $(v_2,R'_i)$ (notice that a connected component may have opposite $\alpha_{i_1}$-direction and $\alpha_{i_2}$-direction for different $i_1$, $i_2$). Then, for any $1 \leq i \leq k-1$, we obtain a group $\alpha'_i$ of $\omega$ vertex-disjoint paths from $S'_i$ to $R'_i$.
\item For any $\breve{\alpha}_k$-path, whose natural $\alpha_k$-direction is from one end vertex, say $v_1$, to the other end vertex, say $v_2$, add two directed edges $(S'_k,v_1)$ and $(v_2,R'_k)$. Then, we obtain a group $\alpha'_k$ of $N_3$ vertex-disjoint paths from $S'_k$ to $R'_k$.
\end{enumerate}

Obviously, the number of global new hubs in $G$ is just $\mathcal{H}(G')$. It follows from the minimality of $G$ and the observation that for any $1 \leq i \leq k$, any $\alpha'_i$-consistent cycle in $G'$ naturally corresponds an $\alpha_i$-consistent cycle in $G$ that $G'$ is a minimal $(\underbrace{\omega,\omega,\ldots,\omega}_{k-1},N_3)$-graph. We then proceed to deduce that $\alpha'_k$ is non-reroutable in $G'$, since otherwise there exists an edge $e$ which is exclusively owned by $\alpha'_k$-paths (this follows from a parallel argument leading to Remark~\ref{ifreroutable}) such that $G' \setminus \{e\}\in \mathcal{N}(\underbrace{\omega,\omega,\ldots,\omega}_{k-1},N_3)$, violating the fact that $G'$ is minimal.

Now, after necessary rerouting of $\alpha'_i$ within $G'$, we assume that $G'_i$, the subgraph of $G'$ consisting of all $\alpha'_i$-paths and $\alpha'_k$-paths, is non-reroutable and thus, by Theorem~\ref{equivalent}, minimal. Notice that any hub $v$ in $G'$ must be a hub of some $G'_i$, since otherwise $G'$ contains an edge incident with $v$ that does not belong to any $G'_i$ and therefore is not minimal. Hence, by Theorem~\ref{Nmn2mn}, we obtain that
$$\mathcal{H}(G')\le \sum_{i=1}^{k-1}\mathcal{H}(G'_i)\le (k-1)\mathscr{H}(\omega,N_3)=2(k-1)\omega N_3\le (k-1)N_1(C_k+N_2).$$
Therefore,
$$\mathcal{H}(G)\le N_1+N_2+(k-1)N_1(C_k+N_2)<\infty.$$
The proof is then complete.
\end{proof}

\begin{thm} \label{thm11mn}
For any $n\ge 0$, $C_1, C_2\ge 1$, we have
$$\mathscr{H}(C_1,C_2,\underbrace{1,1,\ldots,1}_n)=2(C_1C_2+n).$$
\end{thm}

\begin{proof}
The case when $n=0$ is nothing but Theorem~\ref{Nmn2mn}. So we only have to prove the theorem when $n \geq 1$.

{\bf The ``$\leq$'' direction: } We will establish this direction using an inductive argument on $n$. Suppose the inequality holds when $n<t$. Consider a minimal $(C_1,C_2,\underbrace{1,1,\ldots,1}_t)$-graph $G$ with $C_1$ vertex-disjoint paths $\phi_1,\phi_2,\ldots,\phi_{C_1}$ from $S_1$ to $R_1$, $C_2$ vertex-disjoint paths $\psi_1,\psi_2,\ldots,\psi_{C_2}$ from $S_2$ to $R_2$, and a path $\beta_i$ from $S_i$ to $R_i$ for $3\le i\le t+2$. Let $G_1$ be the subgraph of $G$ induced on $$\phi_1,\phi_2,\ldots,\phi_{C_1},\psi_1,\psi_2,\ldots,\psi_{C_2},\beta_3,\beta_4,\ldots,\beta_{t+1}.$$
Now we split $S_1$ into $C_1$ copies $S_1^{(1)}, S_1^{(2)}, \ldots, S_1^{(C_1)}$;  $R_1$ into $C_1$ copies $R_1^{(1)}, R_1^{(2)}, \ldots, R_{1}^{(C_1)}$; $S_2$ into $C_2$ copies $S_2^{(1)}, S_2^{(2)}, \ldots, S_2^{(C_2)}$; $R_2$ into $C_2$ copies $R_2^{(1)}, R_2^{(2)}, \ldots, R_2^{(C_2)}$, such that $\phi_i$ has starting point $S_1^{(i)}$ and ending point $R_1^{(i)}$ for $1\le i\le C_1$; $\psi_i$ has starting point $S_2^{(i)}$ and ending point $R_2^{(i)}$ for $1\le i\le C_2$. Let $G_2$ denote the resulting graph and $\omega$ be the number of weakly connected components (which means connected components when disregarding the orientation) in $G_2$. Now for any $1\le i<j\le C_1$, we identify $S_1^{(i)}$ and $S_1^{(j)}$ if they belong to the same component and we perform similar operations on $R_1^{(i)}$'s, $S_2^{(i)}$'s and $R_2^{(i)}$'s. Let $\widehat{G}$ denote the resulting graph, which consists of $\omega$ connected components. Note that the $i$-th component $\widehat{G}_i$ is in fact a minimal $(C_{1,i},C_{2,i},\underbrace{1,1,\ldots,1}_{t_i})$-graph, where $
\sum_{i=1}^\omega C_{1,i}=C_1$, $\sum_{i=1}^\omega C_{2,i}=C_2$ and $\sum_{i=1}^\omega t_i=t-1$. Notice that for any $i$, at least one of $C_{1,i}$, $C_{2,i}$ and $t_i$ is nonzero, but some $C_{1,i}$'s, $C_{2,i}$'s, $t_i$'s can be zero, for which case, a $(C_{1,i},C_{2,i},\underbrace{1,1,\ldots,1}_{t_i})$-graph can be interpreted as if all zero-valued parameters are simply dropped. For each component $\widehat{G}_i$, by the induction hypothesis, we have

\begin{displaymath}
\mathcal{H}(\widehat{G}_i)\le\begin{cases}
\ 2(C_{1,i}C_{2,i}+t_i) & \textrm{  if }C_{1,i},C_{2,i}\ge 1,\\
\ 2(C_{2,i}+t_i-1) & \textrm{  if }C_{1,i}=0,C_{2,i}\ge 1,\\
\ 2(C_{1,i}+t_i-1) & \textrm{  if }C_{1,i}\ge 1,C_{2,i}=0,\\
\ 2(t_i-1) & \textrm{  if }C_{1,i},C_{2,i}=0.
\end{cases}
\end{displaymath}
Notice that the right hand side of the above inequality can be unified as
$$2\left[C_{1,i}+C_{2,i}+(C_{1,i}-1)^+\cdot(C_{2,i}-1)^++t_i-1\right],$$
for any $C_{1,i},C_{2,i}\ge 0$; here, $(\cdot)^+=\max\{\cdot,0\}$.

A hub in $G$ is said to be {\it new} if it is not a hub in $\widehat{G}$. Notice that the path $\beta_t$ ``meets'' each component at most once, yielding at most two new hubs (more precisely, it meets one edge in $\widehat{G}_i$, then departs from one edge in $\widehat{G}_i$, and thereafter, it will never meet any edge in $\widehat{G}_i$ again), otherwise $G$ is not minimal. Hence,
\begin{align*}
\mathcal{H}(G)
\le & \ 2\omega+\sum_{i=1}^\omega \mathcal{H}(\widehat{G}_i)\\
\le & \ 2\omega+\sum_{i=1}^\omega 2\left[C_{1,i}+C_{2,i}+(C_{1,i}-1)^+\cdot(C_{2,i}-1)^++t_i-1\right]\\
\le & \ 2\omega+2(C_1+C_2+t-1-\omega)+2\left[\sum_{i=1}^\omega(C_{1,i}-1)^+\right]\cdot\left[\sum_{i=1}^\omega(C_{2,i}-1)^+\right]\\
\le & \ 2(C_1+C_2+t-1)+2(C_1-1)(C_2-1)\\
=   & \ 2(C_1C_2+t).
\end{align*}

{\bf The ``$\ge$'' direction: } It suffices to construct a minimal $(C_1,C_2,\underbrace{1,1,\ldots,1}_n)$-graph with $2(C_1C_2+n)$ hubs; see Figure~\ref{picn1122} for an example.
\begin{figure}[h]
\centering
\includegraphics[width=0.48\textwidth]{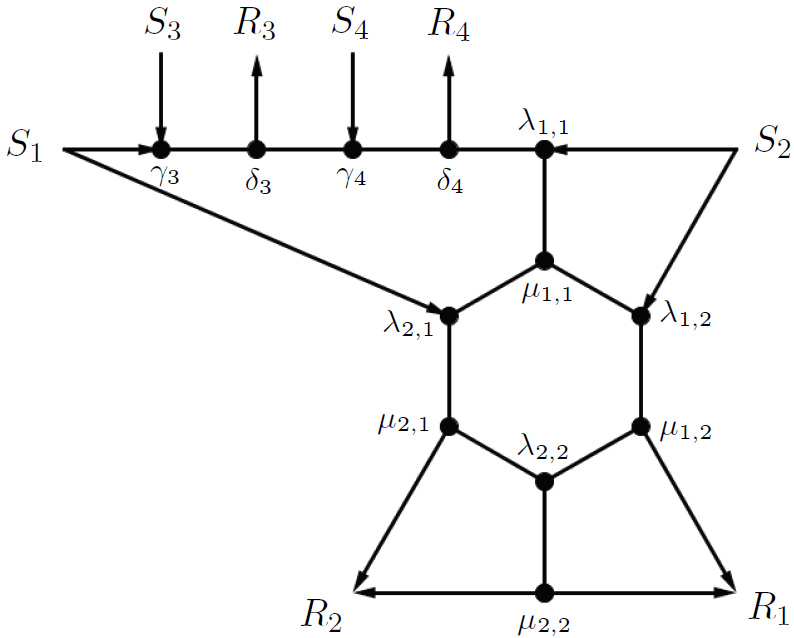}
\caption{A minimal $(2,2,1,1)$-graph with $12$ hubs.}\label{picn1122}
\end{figure}
It turns out the graph $G$, described below in detail, is such a graph: $G$ is naturally orientable; there is $C_1$ vertex-disjoint paths $\phi_1,\phi_2,\ldots,\phi_{C_1}$ from $S_1$ to $R_1$, $C_2$ vertex-disjoint paths $\psi_1,\psi_2,\ldots,\psi_{C_2}$ from $S_2$ to $R_2$, and a path $\beta_i$ from $S_i$ to $R_i$ for $3\le i\le n+2$; paths $\phi_i$ and $\psi_j$ meet at vertex $\lambda_{i,j}$ and depart at vertex $\mu_{i,j}$ for $1\le i\le C_1$ and $1\le j\le C_2$; paths $\beta_i$ and $\phi_1$ meet at vertex $\gamma_i$ and depart at vertex $\delta_i$ for $3\le i\le n+2$. Furthermore, in $\overrightarrow{G}$, we have
\begin{itemize}
\item for $3\le i\le n+2$, on path $\beta_i$,
$$\gamma_i<\delta_i;$$
\item on path $\phi_1$,
$$\gamma_3<\delta_3<\gamma_4<\delta_4<\cdots<\gamma_{n+2}<\delta_{n+2}<\lambda_{1,1}<\mu_{1,1}<\lambda_{1,2}<\mu_{1,2}<\cdots<\lambda_{1,C_2}<\mu_{1,C_2};$$
\item for $2\le i\le C_1$, on path $\phi_i$,
$$\lambda_{i,1}<\mu_{i,1}<\lambda_{i,2}<\mu_{i,2}<\cdots<\lambda_{i,C_2}<\mu_{i,C_2};$$
\item for $1\le j\le C_2$, on path $\psi_j$,
$$\lambda_{1,j}<\mu_{1,j}<\lambda_{2,j}<\mu_{2,j}<\cdots<\lambda_{C_1,j}<\mu_{C_1,j}.$$
\end{itemize}
It can be easily checked that this graph is minimal. The proof is then complete.
\end{proof}

\begin{thm}\label{thmh222}
$$\mathscr{H}(2,2,2)=12.$$
\end{thm}

\begin{proof}

First, it can be verified that the graph in Figure~\ref{n222with12iv} is a minimal $(2,2,2)$-graph with $12$ hubs, which implies that
$$
\mathscr{H}(2,2,2)\ge 12.
$$

\begin{figure}[h]
\centering
\includegraphics[width=0.48\textwidth]{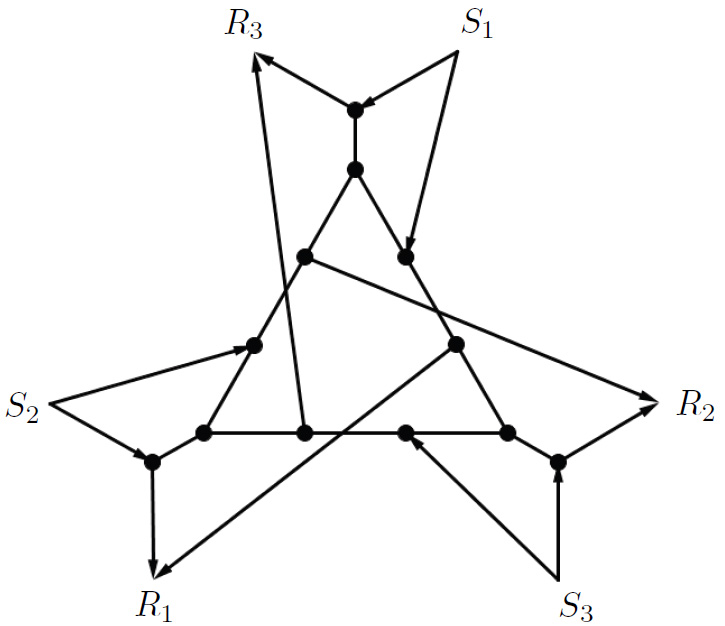}
\caption{a minimal $(2,2,2)$-graph with $12$ hubs}\label{n222with12iv}
\end{figure}

So we only need to prove the other direction. By contradiction, suppose that a minimal $(2,2,2)$-graph $G$ has at least $13$ hubs, and there exists a set of two vertex-disjoint path $\phi=\{\phi_1,\phi_2\}$ from $S_1$ to $R_1$, a set of two vertex-disjoint paths $\psi=\{\psi_1,\psi_2\}$ from $S_2$ to $R_2$, and a set of two vertex-disjoint paths $\xi=\{\xi_1, \xi_2\}$ from $S_3$ to $R_3$.

Here, we note that when there are three pairs of sources and sinks, the equivalence statements as in Theorem~\ref{equivalent} do not hold any more: it turns out that a minimal graph can be reroutable (see Figure~\ref{N222minimalreroutable}).

\begin{figure}[h]
\centering
\includegraphics[width=0.45
\textwidth]{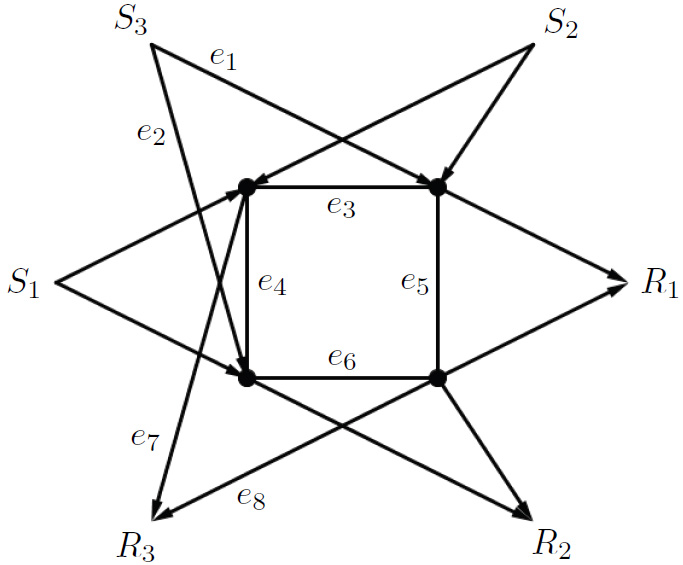}
\caption{A minimal but reroutable $(2,2,2)$-graph: the graph is minimal, but we can choose either $\{e_1\circ e_3\circ e_7, e_2\circ e_6\circ e_8\}$ or $\{e_1\circ e_5\circ e_8, e_2\circ e_4\circ e_7\}$ to be a set two vertex-disjoint paths from $S_3$ to $R_3$.}\label{N222minimalreroutable}
\end{figure}

{\bf We first consider the case $G$ is minimal and non-reroutable.} Let $G_{\phi\psi}$ be the subgraph of $G$ induced on the edges of $\phi$-paths and $\psi$-paths; similarly, we define $G_{\phi\xi}$ and $G_{\psi\xi}$. By Theorem~\ref{equivalent}, these three subgraphs are all minimal, since they are all non-reroutable. Suppose $G_{\phi\psi}$ has the most hubs among them. Every hub in $G$ belongs to at least one of these three subgraphs, so we have
$$13\le \mathcal{H}(G)\le \mathcal{H}(G_{\phi\psi})+\mathcal{H}(G_{\phi\xi})+\mathcal{H}(G_{\psi\xi})\le 3\mathcal{H}(G_{\phi\psi}),$$
and hence $\mathcal{H}(G_{\phi\psi})\ge 5$.

Now, we transform $G_{\phi\psi}$ into a corresponding graph $G_{\phi\psi}^\bullet$ by shrinking each public edge into a vertex (this can be regarded as the ``inverse'' operation of Step 2 in Section~\ref{sectionminimalrepresentation}). In more detail, for a public edge $(v_1,v_2)$, say $e_1=(v_1,u_1), e_2=(v_1,u_3)$ are two private edges incident with $v_1$, and $e_3=(v_2,u_2)$, $e_4=(v_2,u_4)$ are two private edges incident with $v_2$. Then we delete edges $e_1$, $e_2$, $e_3$, $e_4$ and vertices $v_1$, $v_2$, and then add a new vertex $v$ and edges $(v,u_1)$, $(v,u_2)$, $(v,u_3)$, $(v,u_4)$. If $G_{\phi\psi}$ has at least five hubs, then as shown in Figure~\ref{N222threecases}, up to isomorphism, $G_{\phi\psi}^\bullet$ has three possibilities (note that the first and second are different).
\begin{figure}[h]
\centering
\includegraphics[width=1\textwidth]{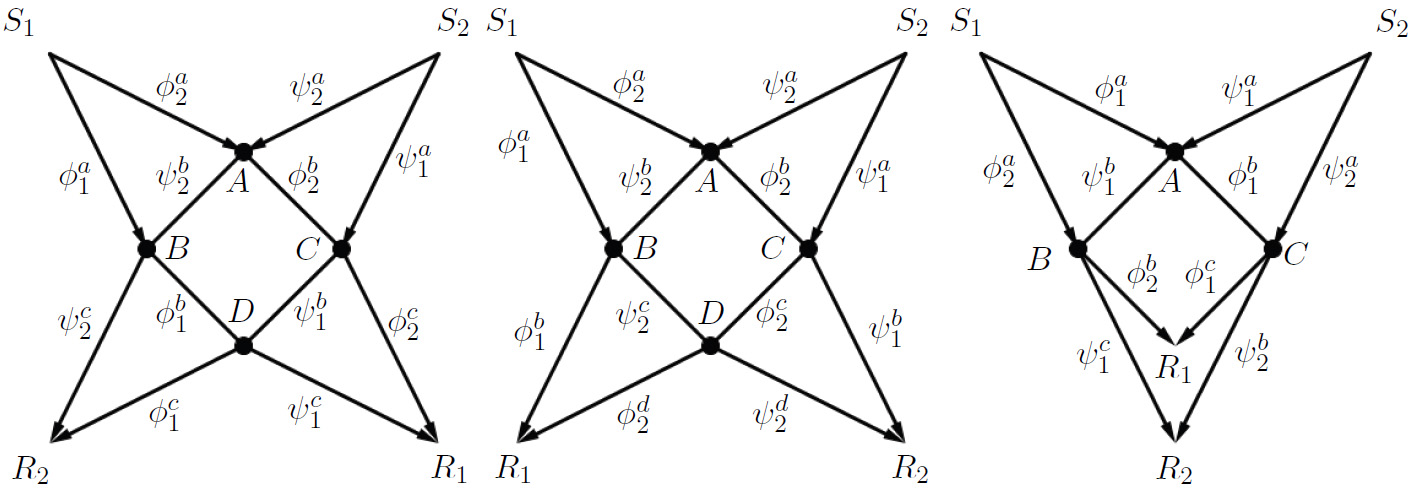}
\caption{Three possibilities of $G_{\phi\psi}^\bullet$}\label{N222threecases}
\end{figure}

Next, we examine in the ways one can add $\xi$-paths into $G_{\phi \psi}$ to form $G$ such that $G$ is minimal. In the following, a hub in $G$ is said to be $\emph{new}$ if it is not a hub in $G_{\phi\psi}$, and an edge is said to be $\emph{new}$ if it does not belong to $G_{\phi\psi}$ and not incident with $S_3$ or $R_3$.

\textbf{Case 1:} This case is shown in Figure~\ref{N222threecases}(a), where each edge is labeled. Since $G_{\phi\psi}^\bullet$ has four hubs, $G_{\phi\psi}$ has at most eight hubs. And since $\mathcal{H}(G)\ge 13$, $\xi_1$, $\xi_2$ have to be added to generate at least $5$ new hubs. Without loss of generality, say, $\xi_1$ contains at least three new hubs. Observe that each of these new hubs is incident with exactly one $\xi_1$-edge that does not belong to any $\phi$ or $\psi$-path, and $S_3$ and $R_3$ are also incident with one such $\xi_1$-edge. So at least $\lceil(3+2)/2\rceil=3$ edges of $\xi_1$ are exclusively owned by $\xi_1$, and thereby at least one of them, say $(v_1, v_2)$, is a new edge (not incident with $S_3$ or $R_3$). We next discuss the possible locations of $(v_1, v_2)$; here, notice that it is possible that either $v_1$ or $v_2$ is not a new hub.

Suppose $v_1\in \phi_2^a$ (this means $v_1$ is one of vertices in $\phi_2^a$, including its two end vertices). Then, one verifies that if $v_2\in \phi_1^b,\phi_1^c,\phi_2^a,\phi_2^b,\phi_2^c,\psi_1^a,\psi_1^b,\psi_1^c,\psi_2^a$ or $\psi_2^b$, then $\phi$ is reroutable; if $v_2\in \phi_1^a$ or $\psi_2^c$, then $\psi$ is reroutable. For example, if $v_2\in \psi_2^c$, then we can find another two vertex-disjoint $\psi$-paths
\begin{align*}
\psi_1'&=\psi_1,\\
\psi_2'&=\psi_2[S_2,A]\circ \phi_2[A,v_1]\circ (v_1,v_2)\circ \psi_2[v_2,R_2]
\end{align*}
from $S_2$ to $R_2$, and thus $G_{\phi\psi}$ is reroutable. And if $v_2\in \phi_1^c$, then we can find another two vertex-disjoint paths
\begin{align*}
\phi_1'&=\phi_1[S_1,B]\circ \psi_2[B,A]\circ \phi_2[A,R_1],\\
\phi_2'&=\phi_2[S_1,v_1]\circ (v_1,v_2)\circ \phi_1[v_2,R_1]
\end{align*}
from $S_1$ to $R_1$, and thus $G_{\phi\psi}$ is reroutable.

Suppose $v_1\in \phi_2^b$. Then, one verifies that if $v_2\in \phi_1^a,\phi_1^c,\phi_2^a,\phi_2^b,\phi_2^c,\psi_1^a,\psi_1^b,\psi_2^a$ or $\psi_2^b$, then $\phi$ is reroutable; and if $v_2\in \phi_1^b,\psi_1^c$ or $\psi_2^c$, $\psi$ is reroutable.

By symmetry, if $v_1\in \phi_1^a,\phi_1^b,\phi_1^c,\phi_2^c,\psi_1^a,\psi_1^b,\psi_1^c,\psi_2^a,\psi_2^b$ or $\psi_2^c$, $G_{\phi\psi}$ is reroutable.

Hence, we only need to check the last possible case: all new edges (including $(v_1,v_2)$) are incident with public edges in $G_{\phi\psi}$. By symmetry, say, $v_1\in A$. Then, if $v_2\in A$ or $B$, $\psi$ is reroutable; if $v_2\in C$, $\phi$ is reroutable; if $v_2\in D$, there must be a new edge $(v_1',v_2')$ of $\xi_2$ with ($v_1'\in B$ and $v_2'\in C$) or ($v_1'\in D$ and $v_2'\in A$) (see Figure~\ref{N222case1and2finalsentence}(a)(b)), since otherwise $G\setminus \{(v_1,v_2)\}$ is still in $\mathcal{N}(2,2,2)$, which contradicts the fact that $G$ is minimal. However, in both cases, we can still find another two vertex-disjoint paths
\begin{align*}
\phi_1'&=\phi_1[S_1,v_1']\circ (v_1',v_2')\circ \phi_2[v_2',R_1],\\
\phi_2'&=\phi_2[S_1,v_1]\circ (v_1,v_2)\circ \phi_1[v_2,R_1]
\end{align*}
from $S_1$ to $R_1$, and thus $G_{\phi\psi}$ is reroutable.

\begin{figure}[h]
\centering
\includegraphics[width=1\textwidth]{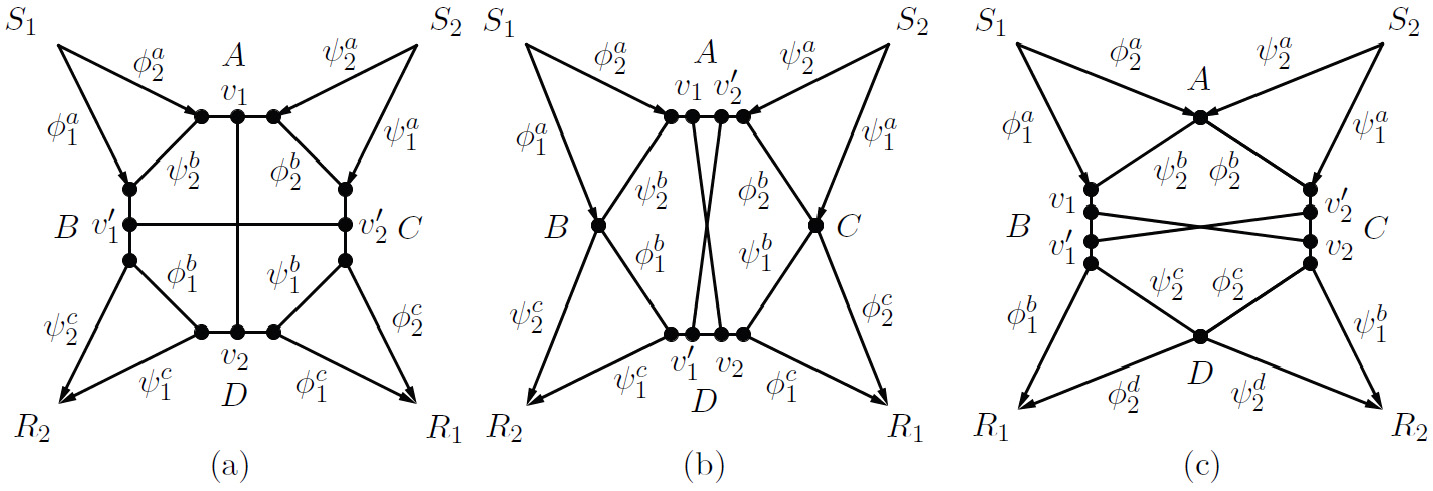}
\caption{The possible locations of the new edge $(v_1',v_2')$.}\label{N222case1and2finalsentence}
\end{figure}

\textbf{Case 2:} This case is shown in Figure~\ref{N222threecases}(b). Similarly as Case 1, we can find a new edge $(v_1,v_2)$ of $\xi_1$.

Suppose $v_1\in \phi_1^a$. Then, one verifies that if $v_2\in \phi_1^a,\phi_1^b,\phi_2^b,\phi_2^c,\phi_2^d,\psi_1^a,\psi_1^b,\psi_2^a,\psi_2^b,\psi_2^c$ or $\psi_2^d$, then $\phi$ is reroutable; if $v_2\in \phi_2^a$, then $\psi$ is reroutable.

Suppose $v_1\in \phi_2^a$. Then, one verifies that if $v_2\in \phi_1^b,\phi_2^a,\phi_2^b,\phi_2^c,\phi_2^d,\psi_1^a,\psi_1^b,\psi_2^a,\psi_2^b,\psi_2^c$ or $\psi_2^d$, then $\phi$ is reroutable; if $v_2\in \phi_1^a$, then $\psi$ is reroutable.

Suppose $v_1\in \phi_2^b$. Then, one verifies that if $v_2\in \phi_1^a,\phi_1^b,\phi_2^a,\phi_2^b,\phi_2^c,\phi_2^d,\psi_1^a,\psi_1^b,\psi_2^a,\psi_2^b,\psi_2^c$ or $\psi_2^d$, then $G_{\phi\psi}$ is reroutable.

By symmetry, if $v_1\in \phi_1^b,\phi_2^c,\phi_2^d,\psi_1^a,\psi_1^b,\psi_2^a,\psi_2^b,\psi_2^c$ or $\psi_2^d$, then $G_{\phi\psi}$ is reroutable.

Hence, we only need to check the last possible case: all new edges (including $(v_1',v_2')$) are incident with public edges in $G_{\phi\psi}$. By symmetry, say $v_1\in A$ or $B$. If $v_1\in A$ and $v_2\in A,C$ or $D$, then $\phi$ is reroutable; if $v_1\in A$ and $v_2\in B$, then $\psi$ is reroutable; If $v_1\in B$ and $v_2\in A,B$ or $D$, $\psi$ is reroutable; If $v_1\in B$ and $v_2\in C$, there must be a new edge $(v_1',v_2')$ of $\xi_2$ with $v_1'\in B$ and $v_2'\in C$ (see Figure~\ref{N222case1and2finalsentence}(c)), since otherwise $G\setminus\{(v_1,v_2)\}$ is still in $\mathcal{N}(2,2,2)$, which contradicts the fact that $G$ is minimal. However, we can still find another two vertex-disjoint paths
\begin{align*}
\phi_1'&=\phi_1[S_1,v_1]\circ (v_1,v_2)\circ \phi_2[v_2,R_1],\\
\phi_2'&=\phi_2[S_1,v_2']\circ (v_2',v_1')\circ \phi_1[v_1',R_1]
\end{align*}
from $S_1$ to $R_1$, and thus $G_{\phi\psi}$ is reroutable.

\textbf{Case 3:} This case is shown in Figure~\ref{N222threecases}(c), where each edge is also labeled. Since $G_{\phi\psi}^\bullet$ has three hubs, $G_{\phi\psi}$ has at most six hubs. And since $\mathcal{H}(G)\ge 13$, $\xi_1$, $\xi_2$ have to be added to generate at least $7$ new hubs. So at least $\lceil(7+2\times 2)/2\rceil=6$ edges of $\xi_1$ and $\xi_2$ do not belong to any $\phi$ or $\psi$-path, and thereby at least two of them, say, $(v_1,v_2)$ and $(v_1',v_2')$ are two new edges. We next discuss their possible locations.

Suppose $v_1\in \phi_1^a$. Then, one verifies that if $v_2\in \phi_1^a,\phi_1^b,\phi_1^c,\phi_2^b,\psi_1^a,\psi_1^b,\psi_2^a,\psi_2^b,A$ or $C$, then $\phi$ is reroutable; if $v_2\in \phi_2^a,\psi_1^c$ or $B$, then $\psi$ is reroutable.

Suppose $v_1\in \phi_1^b$. Then, one verifies that if $v_2\in \phi_1^a,\phi_1^b,\phi_1^c,\phi_2^a,\psi_1^a,\psi_1^b,\psi_2^a,\psi_2^b,A$ or $C$, then $\phi$ is reroutable; if $v_2\in \phi_2^b,\psi_1^c$ or $B$, then $\psi$ is reroutable.

Suppose $v_1\in \phi_2^a$. Then, one verifies that if $v_2\in \phi_1^b,\phi_1^c,\phi_2^a,\phi_2^b,\psi_1^b,\psi_1^c,\psi_2^a,\psi_2^b,B$ or $C$, then $\phi$ is reroutable; if $v_2\in \phi_1^a,\psi_1^a$ or $A$, then $\psi$ is reroutable.

Suppose $v_1\in A$. Then, one verifies that if $v_2\in \phi_1^a,\phi_1^b,\phi_1^c,\psi_1^a,\psi_1^b,\psi_2^a,\psi_2^b,A$ or $C$, then $\phi$ is reroutable; if $v_2\in \phi_2^a,\phi_2^b,\psi_1^c$ or $B$, then $\psi$ is reroutable.

By symmetry, if $v_1\in \psi_1^a,\psi_1^b$ or $\psi_2^a$, then $G_{\phi\psi}$ is reroutable.

Let $G_B$ be the subgraph of $G$ induced on the edges of $\psi_1[B,R_2]$ and $\phi_2[B,R_1]$ and $G_C$ be the subgraph induced on the edges of $\phi_1[C,R_1]$ and $\psi_2[C,R_2]$. If $v_1, v_2\in G_B$, either $\phi$ or $\psi$ is reroutable, and hence $v_1\not \in G_B$ or $v_2\not \in G_B$. Similarly, we deduce that $v_1\not \in G_C$ or $v_2\not \in G_C$. Then we just need to check the last possibility: $v_1,v_1'\in G_B$ and $v_2,v_2'\in G_C$, for which we can also easily rule out the subcases when (1) $v_1\in \phi_1$ and $v_1'\in \psi_2$ (2) $v_2\in \phi_2$ and $v_2'\in \psi_1$ (3) $v_1,v_1'\in B$ (4) $v_2,v_2'\in C$ (5) $v_1,v_1'\in \phi_2$ and $v_2,v_2'\in \phi_1$ (6) $v_1,v_1'\in \psi_1$ and $v_2,v_2'\in \psi_2$. So, in the following, we examine the remaining two subcases:

(1) $v_1,v_1'\in \phi_2$ and $v_2,v_2'\in \psi_2$. By symmetry, we assume that $v_1<v_1'$ on $\phi_2$ in $\overrightarrow{G}_{\phi\psi}$. If $v_2<v_2'$, then $\xi$ is reroutable; If $v_2>v_2'$, then $\phi$ or $\xi$ is reroutable.

(2) $v_1,v_1'\in \psi_1$ and $v_2,v_2'\in \phi_1$. By symmetry, we assume $v_1<v_1'$ on $\psi_1$ in $\overrightarrow{G}_{\phi\psi}$. If $v_2<v_2'$, then $\xi$ is reroutable; If $v_2>v_2'$, then $\psi$ is reroutable.
Now, we are ready to conclude that any minimal and non-reroutable graph in $\mathcal{N}(2,2,2)$ has at most $12$ hubs.

{\bf We next consider the case when $G$ is a minimal but reroutable $(2,2,2)$-graph}. Suppose $\xi$ is reroutable in the subgraph $G_{\psi\xi}$. By Theorem~\ref{equivalent}, there exists a $\xi$-consistent cycle
$$p_1 \circ e_1 \circ p_2 \circ e_2\circ \cdots \circ p_d\circ e_d,$$
where each $p_i$ is a subpath of some $\xi$-path and each $e_i$ is a private $\psi$-edge (in $G_{\psi \xi}$). We further assume that $d$ (the number of private $\psi$-edges) is the smallest among all possible $\xi$-consistent cycles in $G_{\psi \xi}$. Then, each $e_i$ belongs to a different $\psi$-path, which implies $d\le 2$. Since $G$ is minimal, each $\xi$-private edge (in $G_{\psi \xi}$) in $p_1,p_2,\ldots,p_d$ must belong to $\phi$-paths as well. So, each edge of the $\xi$-consistent cycle in $G$ belongs to either some $\phi$-path or $\psi$-path. Note that for any edge $e$ of this cycle, $G_{\phi\psi} \backslash \{e\}$ does not belong to $\mathcal{N}(C_1, C_2)$, since otherwise, together with the observation that $\xi$ can be rerouted using edges in $G \backslash \{e\}$, we deduce that $G \backslash \{e\} \in \mathcal{N}(C_1, C_2, C_3)$, which contradicts the minimality of $G$. Then we can find a minimal subgraph $\widehat{G}_{\phi\psi}$ of $G_{\phi\psi}$, which is of either Case 1 (when $d=2$) or Case 2 (when $d=1$). Let $G_\xi$ be the subgraph induced on the edges of $\xi$-paths. Then $G=\widehat{G}_{\phi\psi} \cup G_\xi$, by the minimality of $G$. Notice that all contradictions in Case 1 and Case 2 arise from the rerouting of $\phi$ or $\psi$-paths. So we can follow similar arguments and conclude $\mathcal{H}(G)\le 12$. The proof is then complete.
\end{proof}

\noindent {\bf Acknowledgements:} We are indebted to Weiping Shang for insightful discussions during the early stage of this work.

\end{document}